\documentclass[12pt]{article}
\usepackage{latexsym, amsmath, amsfonts, amsthm, amssymb}
\usepackage{times}
\usepackage{a4wide}
\usepackage{times}
\usepackage{graphicx, tocloft}

\def \RR {\mathbb R}

\def \cF {\mathcal F}

\newtheorem{theorem}{Theorem}

\newtheorem{lemma}[theorem]{Lemma}

\newtheorem{proposition}[theorem]{Proposition}

\newtheorem{remark}[theorem]{Remark}

\def\myffrac#1#2 in #3{\raise 2.6pt\hbox{$#3 #1$}\mkern-1.5mu\raise 0.8pt\hbox{$
#3/$}\mkern-1.1mu\lower 1.5pt\hbox{$#3 #2$}}

\begin{document}

\title{One more proof of the Alexandrov-Fenchel inequality}
\author{D. Cordero-Erausquin, B. Klartag, Q. Merigot, F. Santambrogio}
\date{}
\maketitle

\setcounter{section}{1}

\medskip
This short text has the modest goal of presenting another proof of the Alexandrov-Fenchel inequality for mixed volumes. It is similar to the proof for polytopes discussed  in the recent work of Shenfeld and van Handel~\cite{SvH} which is mainly devoted to a short and neat treatment of the  case of smooth bodies.  Our presentation puts the emphasis on the basic algebraic properties of the polynomials involved in the construction of mixed volumes. Actually, once elementary and classical geometric and algebraic properties have been recalled, our argument reduces to the simple  Proposition~\ref{main-prop} below.

\medskip
The history of the Alexandrov-Fenchel inequality and its various proofs until the 1980s is described in the book by Burago and Zalgaller \cite[Section 20.3]{BZ}.
The
more recent literature contains a proof by Wang \cite{W} that was inspired by Gromov's work \cite{Gro}, in addition to the proof by  Shenfeld and van Handel \cite{SvH}.
%A proof using a geometric flow  was proposed by Andrews in his Ph.D. thesis \cite{Andrews}.
Applications to combinatorics are described by Stanley \cite{St}.

%. It is inspired by G\aa rding's theory of hyperbolic polynomials \cite{garding}, though this theory is not directly applied here beyond the trivial case of quadratic polynomials.

\medskip
The Minkowski sum of two sets $A, B \subseteq \RR^n$
is $A + B = \{ a + b \, ; \, a \in A, b \in B \}$, and we also write $t A = \{ t x \, ; \, x \in A \}$ for $t \in \RR$. Minkowski has shown that when $K_1,\ldots, K_N \subseteq \RR^n$ are convex bodies, the function
\begin{equation}  \RR^N_+ \ni (t_1,\ldots, t_N) \mapsto Vol_n \left( \sum_{i=1}^N t_i K_i \right) \label{eq_849} \end{equation}
is a homogeneous polynomial of degree $n$. Particular cases were discussed earlier by the 19th century geometer Jacob Steiner.
Here,
$\RR^N_+ = \{ x \in \RR^N \, ; \, \forall i, x_i > 0 \}$, a convex body is a compact, convex set with a non-empty interior and $Vol_n$ is $n$-dimensional volume.

\medskip
The {\it mixed volumes} are defined as the coefficients of the homogeneous polynomial in (\ref{eq_849}). Specifically, the mixed volume of $n$ compact, convex sets $K_1,\ldots,K_n \subseteq \RR^n$ is
\begin{align*}  V(K_1,\ldots,K_n)
%& = \frac{1}{n!} \sum_{k=1}^n (-1)^{n+k}
%\sum_{i_1 < \ldots < i_k} Vol_n( K_{i_1} + \ldots + K_{i_n} ) \\ &
 = \frac{1}{n!} \left. \frac{\partial^n}{\partial t_1 \partial t_2 \ldots \partial t_n} Vol_n \left( \sum_{i=1}^n t_i K_i \right) \right|_{t_1=\ldots=t_n = 0} \label{eq_851} \end{align*}
The reader is referred to Schneider's book \cite[Chapter 5]{S} for
explanations and for the basic properties of the mixed volumes.
The mixed volume $V(K_1,\ldots,K_n)$ is symmetric (invariant under permuting the convex bodies),
 continuous in the Hausdorff metric, it satisfies $V(K,\ldots,K) = Vol_n(K)$, it is multi-linear with respect to Minkowski addition, and it is invariant under translating the convex bodies. It turns out that $V(K_1,\ldots,K_n)$ is always non-negative and monotone increasing (with respect to inclusion) in each of its arguments. When $K_1,\ldots,K_n \subseteq \RR^n$ have a non-empty interior, their mixed volume is in fact positive.

 \medskip
 The Alexandrov-Fenchel inequality from the 1930s states that
\begin{equation}
V(K_1, K_2, K_3,\ldots,K_n)^2 \geq
V(K_1, K_1, K_3,\ldots,K_n) V(K_2, K_2, K_3,\ldots,K_n) \label{eq_859}
\end{equation}
for any convex bodies $K_1,\ldots,K_n \subseteq \RR^n$. In other words, this inequality states that the function
\begin{equation}  t \to V(K_1 + t K_2, K_1 + t K_2, K_3,\ldots,K_n) \qquad \qquad (t > 0) \label{eq_1125}
\end{equation}
is a quadratic polynomial with a non-negative discriminant.
Alternatively, the Alexandrov-Fenchel inequalities express  inequalities 
of Brunn-Minkowski type for mixed volumes. Given $m=2,3, \ldots, n$, and $(n-m+2)$ convex bodies $K_{1}, \ldots, K_{n-m+2}\subset \RR^n$, 
it follows from the Alexandrov-Fenchel inequality that the function
\begin{equation}\label{m-concavity}
t \longrightarrow f_m(t):=V(\underbrace{K_{1}+tK_2, \ldots, K_{1}+tK_2}_{m \textrm{ times}}, K_{3}, \ldots, K_{n-m+2})^{1/m} \quad \textrm{ is concave on $\RR^+$}.
\end{equation}
The inequality $f_m ''(0)\le 0$ in the case  $m=2$ is precisely equivalent to the Alexandrov-Fenchel inequality~\eqref{eq_859}, and it easily implies the other cases. The case $m=n$ is the Brunn-Minkwoski inequality for convex sets.

\medskip
We will prove, by induction on the dimension $n$, that the functions $f_m$ are concave. To be precise, assuming the Alexandrov-Fenchel inequalities in dimension $\le n-1$, we will prove that in dimension $n$ and for $m\ge 3$, the function $f_m$ is concave. Crucial but elementary properties of hyperbolic polynomials allow us to proceed to the desired case $m=2$, namely:

\begin{lemma}[``$m=3\Rightarrow m=2$"]\label{m32}
In order to prove the
Alexandrov-Fenchel inequality
\eqref{eq_859}
in $\RR^n$, it suffices to establish property~\eqref{m-concavity} with $m=3$ for all convex bodies in $\RR^n$.
\end{lemma}

\begin{proof} We rely on an elementary linear algebra statement from the
	Appendix of H\"ormander's book~\cite{H1}. Introduce  the symmetric $3$-linear form $\tilde f$ on $\RR^3$, defined on the cone $(\RR_+)^3$ by
	$$\tilde f(x,y,z) = V(x_1 K_1 + x_2 K_2 + x_3 K_3\, , \, y_1 K_1 + y_2 K_2 + y_3 K_3\, , \, z_1 K_1 + z_2 K_2 + z_3 K_3, K_4, \ldots , K_n)$$
	for $x,y,z\in (\RR_+)^3$, and the associated $3$-form $f(x) = \tilde f(x,x,x)$, $x\in \RR^3$. Our assumption is that for $x,y \in (\RR_+)^3$, the function $t \to f(x+ ty)^{1/3}$ is concave on $\RR^+$, which is property $(i)'$ of~\cite[Proposition A1]{H1}. This implies property $(iii)$
of~\cite[Proposition A1]{H1} which gives that
	$$\tilde f(x,y,z)^2 \ge \tilde f(x,x,z) \tilde f(y,y,z), \qquad \forall x,y,z \in (\RR_+)^3$$
	and in turn implies (and is equivalent) to the Alexandrov-Fenchel inequalities~\eqref{eq_859}, for instance by picking (at the limit) $x=(1,0,0)$, $y=(0,1,0)$ and $z=(0,0,1)$.
\end{proof}

%We will include an alternate proof below.

As we said, we prove the Alexandrov-Fenchel inequality (\ref{eq_859}) by induction on the dimension $n$. The base case of our induction is the case $n=2$, in which it is well-known that
the
Alexandrov-Fenchel inequality follows from the Brunn-Minkowski inequality. We do not provide here an alternative proof for the
Alexandrov-Fenchel inequality in two dimensions.

\medskip
Thus, assume that we are given $n \geq 3$, and that the Alexandrov-Fenchel inequality is already proven in dimension $n-1$.
Our goal is to prove that, given $m\ge 3$ and $m+n-2$ convex bodies, the function $f_m$ above is concave (the case $m=3$ suffices).
By translating the convex bodies, it suffices to prove (\ref{m-concavity}) under the additional assumption that the origin belongs to the relative interior of the $K_i$.
By continuity with respect to the Hausdorff distance, it suffices to prove (\ref{m-concavity}) in the case where $K_1, \ldots, K_{n-m+2} \subseteq \RR^n$ are polytopes containing the origin in their interior.
 In fact, according to \cite[Theorem 2.4.15]{S} we may even assume that these polytopes are {\it simple strongly isomorphic}.

 \medskip Let us briefly explain the definition and the basic properties of simple, strongly-isomorphic polytopes, using Schneider \cite[Chapter 2]{S} as our main reference.  A polytope in $\RR^n$ is simple if each of its vertices is contained in exactly $n$ facets.
 The polytopes $K_1,\ldots,K_{n-m+2}$ are strongly isomorphic if  for any $\theta \in S^{n-1}$, the (affine) dimension of the convex sets
 $$ \{ x \in K_i \, ; \, \langle x, \theta \rangle = \sup_{y \in K} \langle y, \theta \rangle \}
 $$
 is the same for all $i=1, \ldots, n-m+2$. Here $S^{n-1} = \{ x \in \RR^n \, ; \, |x| = 1 \}$ is the unit sphere in $\RR^n$, and we write $\langle x, y \rangle$ for the standard scalar product between $x,y \in \RR^n$.

\medskip
 Let
$ u_1,\ldots, u_N \in S^{n-1} $ be the list
of
 outer unit normals to the facets of $K_1$, say. Then this list is also the list of outer unit normals to the facets of $K_i$, for all $i$.
For $h \in \RR^N_+$ we consider the set
$$ K[h] = \left \{ x \in \RR^n \, ; \, \langle x, u_j \rangle \leq h_j \ \text{for} \ j=1,\ldots, N\right \}, $$
which is a convex body containing the origin. 
Denote
$$ C = \{ h \in \RR^N_+ \, ; \, K[h] \textrm{ is strongly isomorphic to } K_1 \}. $$
It is well-known, and stated in the next lemma, that $C$ is a convex cone in $\RR^N$ which provides a parameterization
of the space of all polytopes that are strongly isomorphic to $K_1$ and with zero in their interior. Moreover, this parameterization is linear with respect
to the Minkowski sum: 

\begin{lemma} We have the following properties:
\label{addition} 
\begin{enumerate}
\item[(i)] The set $C \subseteq \RR^N_+$ is an open convex cone with an apex at zero.
\item[(ii)] Each polytope strongly isomorphic to
$K_1$ and containing the origin in its interior takes the form $K[h]$ for a certain uniquely determined $h \in C$.
\item[(iii)] For any $h, h' \in C$,
\begin{equation} K[h + h'] = K[h] + K[h']. \label{eq_1730} \end{equation}
\end{enumerate}
\label{lem_237}
\end{lemma}

\begin{proof} Begin with the proof of (ii). Each polytope strongly isomorphic to
$K_1$ has 
$u_1,\ldots,u_N$ as the unit outer normal to its facets. If such polytope contains 
 the origin in its interior, then it takes the form $K[h]$ for a certain uniquely-determined $h \in \RR^n_+$. By definition, $h \in C$, and
(ii) is proven. Next, it is clear that $t h \in C$ whenever $t > 0$ and $h \in C$. We need to prove that $C$ is open and convex and that (\ref{eq_1730}) holds true. First, note that when $K[h]$ is strongly isomorphic to $K_1$, the polytope $K[h]$ has a facet whose outer normal is $u_j$, and hence
	$$ h_j = \sup_{x \in K[h]} \langle x, u_j \rangle \qquad \qquad \text{for all} \ j. $$
	Given $h, h' \in C$, the convex set
	$$ K[h] + K[h'] $$
	is strongly isomorphic to $K_1$, according to  \cite[Corollary 2.4.12]{S}. It follows from (ii) that $K[h] + K[h'] = K[h'']$ for a certain $h'' \in C$.
	In order to show that $h'' = h + h'$, we note that
	\begin{equation*}
	 h''_j = \sup_{x \in K[h] + K[h']} \langle x, u_j \rangle = \sup_{x \in K[h]} \langle x, u_j \rangle + \sup_{x \in K[h']} \langle x, u_j \rangle = h_j + h'_j. \end{equation*}	
Thus $h'' = h + h' \in C$, and $C$ is convex. Moreover, $K[h] + K[h'] = K[h + h']$ and (iii) is proven.  The fact that $C$ is open follows from the simplicity of the polytope $K[h]$ for any $h \in C$, see \cite[Lemma 2.4.13]{S}. This completes the proof of (i), and the lemma is proven. 
\end{proof}

We conclude from Lemma \ref{lem_237} that each $K_i$ takes the form $K_i = K[h]$ for a certain $ h = h^{(i)} \in \RR^N_+$. We will study the function
$$\mathcal F(h):= V(\underbrace{K[h], \ldots, K[h]}_{m \textrm{ times}}, K_{3}, \ldots, K_{n-m+2}), \qquad h\in C \subset \RR^N.$$
With Lemma \ref{lem_237} in hand, it is clear that when working with our simple strongly isomorphic polytopes $K_i$, the concavity of  $f_m$ on $\RR^N_+$ is \emph{equivalent} to the concavity of the function $\cF^{1/m}$ on the cone $C$, since
$$ f_m(t) = \cF^{1/m}(h^{(1)}+th^{(2)}).$$
For proving that $f_m$ is concave on $\RR^N_+$, it is enough to prove that the function   $\mathcal F^{1/m}$ is concave on $C$. We will do it locally by analyzing its Hessian.

\medskip
A crucial property of $\mathcal F$ is that it is homogeneous of degree $m$ on $C$; actually, $\mathcal F$ may be extended to a homogeneous polynomial of degree $m$ in $\RR^N$ as is explained in \cite[Chapter 5]{S}. Other properties that follow from the geometry of polytopes  will allow us to perform induction on the dimension.
For a convex body $K \subseteq \RR^n$ we write
$$ F_j(K) = \left \{ x \in K \, ; \, \langle x, u_j \rangle = \sup_{y \in K} \langle y, u_j \rangle \right \} $$
for the facet whose normal is $u_j$. In our case, when $K = K_i$, the facets $F_1(K),\ldots, F_N(K)$ have positive $(n-1)$-dimensional volume. It is well-known (e.g. follows from~\cite[Lemma 5.1.5]{S}) that
$$ \partial^i \mathcal F := \frac{\partial \mathcal F}{\partial h_i} =  \frac{m}{n} \cdot V(\underbrace{F_i(K[h]), \ldots, F_i(K[h])}_{m-1 \textrm{ times}}, F_i(K_{3}), \ldots, F_i(K_{n-m+2})) $$
where this is a mixed volume of $(n-1)$ bodies in dimension $n-1$. The facets satisfy the property of Lemma~\ref{addition}, that is, for $h,h'\in C$,
$$F_i(K[h+h'])=F_i(K[h]) + F_i(K[h']).$$
This follows from Lemma~\ref{addition} and \cite[Theorem 1.7.5]{S}. Therefore, by the induction hypothesis (namely the concavity of a certain function $t\to f_{m-1}(t)$ for convex bodies in $\RR^{n-1}\simeq u_i^\perp$), the function $(\partial^i\mathcal F)^{1/{(m-1)}}$ is concave on $C$, and so in particular, $\partial^i\mathcal F$ is log-concave on $C$.
 Similarly
$$ \partial^{ij} \mathcal F := \frac{\partial^2 \mathcal F}{\partial h_i \partial h_j} = c_{ij} \cdot V(\underbrace{F_{ij}(K[h]), \ldots, F_{ij}(K[h]) }_{m-2 \textrm{ times}}, F_{ij}(K_3),\ldots, F_{ij}(K_{n-m+2})) $$
where $F_{ij}(K) = F_i(K) \cap F_j(K)$ and where $c_{ij} = (m/n) \cdot ((m-1)/(n-1)) / \sqrt{1 - \langle u_i, u_j \rangle^2}$ is a non-negative coefficient. Note that $\partial^{ij} \mathcal F$ is non-negative for all $i \neq j$. Moreover, $\partial^{ij} \mathcal F(h) > 0$ whenever the facets $F_i(K[h])$ and $F_j(K[h])$ intersect in an $(n-2)$-dimensional face. We conclude that the Hessian matrix $\nabla^2 \mathcal F$ is  {\it irreducible:}
%has sufficiently many non-zero entries, in the sense that it is
For any two indices $i,j$ there is a chain of indices $i, i_1,\ldots, i_L, j$ such that $\partial^{i i_1 } \cF(h) > 0, \partial^{i_1 i_2} \cF(h) > 0,\ldots, \partial^{i_L j} \cF(h) > 0$.
The conclusion now follows from the following simple abstract Proposition (with $p=m$ and $f=\mathcal F$), which is the core of our argument.

\begin{proposition}\label{main-prop}
 Let $C \subseteq \RR^N_+$ be an open, convex cone
	 and let $p \in (2,+\infty)$.
	Let $f: C \rightarrow (0, \infty)$ be a smooth function such that	\begin{enumerate}
		\item[(i)] The function $f$ is $p$-homogeneous.
		\item[(ii)] $\partial^i f$ is positive and log-concave for $1 \leq i \leq N$ (in other words, $\log (\partial^i f)$ is concave).
		\item[(iii)] $\partial^{ij} f$ is non-negative for all distinct $i,j \in \{1,\ldots,N\}$. Moreover, sufficiently many of these numbers are strictly positive, so that the Hessian matrix $\nabla^2 f$ is irreducible.
		\end{enumerate}	
	Then the function $f^{1/p}$ is concave on the convex cone $C$.
	\label{cor_1142}
\end{proposition}

\begin{proof} For any $i$, the function $f_i = \partial^i f$ is log-concave and $(p-1)$-homogeneous, hence $f_i^{1/(p-1)}$ is concave and so
	\begin{equation} \nabla^2 f_i \leq \frac{p-2}{p-1} \frac{\nabla f_i \otimes \nabla f_i}{f_i}. \label{eq_1230} \end{equation}
	Note that $\sum_i x_i \nabla^2 f_i = \sum_i x_i \partial_i (\nabla^2 f) = (p-2) \nabla^2 f$, again by homogeneity.
	By multiplying (\ref{eq_1230}) by $x_i$ and summing over $i$,
\begin{equation}\label{d-b}
 (p-2) \nabla^2 f = \sum_i x_i \nabla^2 f_i \leq \frac{p-2}{p-1} \sum_{i=1}^N x_i \frac{\nabla f_i \otimes \nabla f_i}{f_i}
	= \frac{p-2}{p-1} (\nabla^2 f) D (\nabla^2 f)
\end{equation}
	where $D$ is defined via
	\begin{equation*} D = {\rm diagonal}( x_1/f_1,\ldots,x_N/f_N). \label{eq_1232} \end{equation*}
	 Since $D$ has positive entries, we may multiply from the left and right by $D^{1/2}$, hence
\begin{equation}\label{d-b-2}
 (p-2) D^{1/2} (\nabla^2 f) D^{1/2} \leq \frac{p-2}{p-1} \left[ D^{1/2} (\nabla^2 f) D^{1/2} \right]^2.
 \end{equation}
	We conclude that the matrix $M = D^{1/2} (\nabla^2 f) D^{1/2}$ thus has no spectrum in the interval $(0, p-1)$. The matrix $M$ has non-negative off-diagonal entries, it is irreducible, and it has an eigenvector $D^{-1/2} x$ with non-negative entries, corresponding to the eigenvalue $p-1$.
		By the Perron-Frobenius theorem, this is the simple, maximal eigenvalue; actually, since we are dealing with a symmetric matrix, this can also be proved by elementary computations without using the Perron-Frobenius theorem.
 Hence,
	$$ M \leq (p-1) \frac{(D^{-1/2} x) \otimes (D^{-1/2} x)}{|D^{-1/2} x|^2} = (p-1) \frac{D^{-1/2} (x \otimes x) D^{-1/2}}{p \cdot f}. $$
	Multiplying by $D^{-1/2}$ on the left and on the right, we get as desired
	\begin{equation*} \nabla^2 f \leq  (p-1) \frac{D^{-1} (x \otimes x) D^{-1}}{p \cdot f} = \frac{p-1}{p} \frac{\nabla f \otimes \nabla f}{f}. \qedhere \end{equation*}
\end{proof}

\begin{remark}{\rm
Properties~\eqref{d-b}-\eqref{d-b-2} for our homogeneous function play the role of a second order integration by parts formula ({\sl \`a la} Bochner). }
\end{remark}

\begin{remark} {\rm Had we had the Proposition in the case $p=m=2$,  we could have deduced the Alexandrov-Fenchel inequalities directly without having to use Lemma~\ref{m32}. For this, one needs to find a proper replacement for the condition $(ii)$, which as such seems too weak  in the case $p=2$ (the function $\partial^i f$ is linear). We leave this as question to the interested reader. }
\end{remark}

{\it Acknowledgement.} We thank Ramon van Handel for his comments on an earlier version of this text.

\bigskip
\bigskip

{\small
 DCE: Institut de Math\'ematiques de Jussieu, Sorbonne Universit\'e, 4 place Jussieu, 75252 Paris, France.

\medskip
BK: Department of Mathematics, Weizmann Institute of Science, Rehovot 76100, Israel.
	
\medskip
QM: Laboratoire de Math\'ematiques d'Orsay, Univ. Paris-Sud, CNRS, Universit\'e
Paris-Saclay, 91405 Orsay, France.

\medskip
FS: Institut Camille Jordan, Universit\'e Claude Bernard - Lyon 1, 43 boulevard du 11 novembre 1918, 69622 Villeurbanne cedex France.
}
\end{document}